\documentclass{amsart}
\usepackage{amssymb,amsmath,latexsym}
\usepackage{amsthm}
\usepackage{fontenc}
\usepackage{amssymb}
\numberwithin{equation}{section}

\newtheorem{theorem}{Theorem}[section]

\newtheorem{lemma}[theorem]{Lemma}

\begin{document}
\author{Alexander E Patkowski}
\title{A general double sum identity, mock theta functions, and Bailey pairs}

\maketitle
\begin{abstract} We obtain some Bailey pairs associated with indefinite quadratic forms with the $\beta_n$ connected to a finite sum. A new general identity is given, which provides identities for $q$-hypergeometric series, including mock theta functions.\end{abstract}



\section{Introduction}
An area of continued development in $q$-hypergeometric series is obtaining Bailey pairs associated with indefinite quadratic forms. One of the primary motivators is to find further identities for Mock theta functions, holomorphic parts of particular weight $\frac{1}{2}$ harmonic Maass Forms [18]. Some of the substantial developments in this way include Andrews [1, 3], Zwegers [19], and Hickerson and Mortenson [11], where mock theta functions identities were developed. Crucial in this framework was use of Bailey pairs to associate these functions with indefinite quadratic forms. For example, in [1] we find the fifth order mock theta function identity
$$\sum_{n\ge 0}\frac{q^{n^2}}{(-q;q)_{n}}=\frac{1}{(q;q)_{\infty}}\sum_{n\ge0}q^{n(5n+1)/2}(1-q^{4n+2})\sum_{|j|\le n}(-1)^{j}q^{-j^2},$$
where as usual [8] $(w;q)_n=(w)_n:=(1-w)(1-wq)\cdots(1-wq^{n-1}).$ Many other authors have utilized Bailey pairs [3, 7, 10, 15], as well as created multisums related to Mock theta functions [10, 13, 14, 16].
\par The goal of the present paper to to consider certain double sums which involve the truncated form of the sum [8, pg.4, eq.(6.2)]
\begin{equation}\sum_{n\ge0}\frac{(a;q)_nz^n}{(q;q)_n}=\frac{(az;q)_{\infty}}{(z;q)_{\infty}}.\end{equation} Specifically, the finite sum considered by Fine [8, pg.17, eq.(15.4)].
To illustrate an interesting motivating example, recall the identity [8, pg.17, eq.(15.52)] 
\begin{equation}(z;q)_{\infty}\sum_{n\ge0}(aq;q)_nz^n=\sum_{n\ge0}(-z)^nq^{n(n+1)/2}\sum_{0\le k \le n}\frac{a^{n-k}}{(q)_{k}}.\end{equation}
Taking the limit $z\rightarrow1$ and setting $a=1$ we have
\begin{equation}\sum_{n\ge0}(-1)^nq^{n(n+1)/2}\sum_{0\le k \le n}\frac{1}{(q)_{k}}=(q)_{\infty}^2.\end{equation}
The product on the right side of (1.3) is a weight one modular form that has a connection to $\mathbb{Q}(\sqrt{3})$ [12], and has an indefinite quadratic form expansion due to L.J. Rogers. To interpret the left had side as a generating function we need some notation for partitions. Let $\#(\pi)$ be the number of parts of a partition $\pi,$ $L(\pi)=\pi_1$ be the largest part of a partition $\pi,$ and we say $\pi$ is a partition of $n$ if $\sum_{1\le i \le r}\pi_i=n$ and $\pi_1\ge \pi_2 \ge \cdots \ge \pi_r.$ \par Consider for $r\le N,$
\begin{equation}\frac{q^{N(N+1)/2}}{(q;q)_r}.\end{equation}
Clearly (1.4) generates a partition pair $(\lambda, \mu)$ where $\lambda$ is a triangular partition with $\#(\lambda)=L(\lambda)=N,$ and $\mu$ is a partition with $L(\mu)\le r.$ We take $\mu$ and enlarge it by creating a partition $\tau$ where $L(\tau)=\tau_1=L(\lambda)+L(\mu),$ the next part $\tau_2=N-1+\mu_2,$ and so on so that $\tau_i=N-(i-1)+\mu_i,$ for $1\le i\le N.$ (In the case $r=N,$ $\tau$ would be a partition into $N$ distinct parts.) Taking the conjugate partition, we create a $1$-to-$1$ correspondence to a partition $\varpi$ where $L(\varpi)=\max\{\#(\lambda), \#(\mu)\},$ and $\#(\varpi)=L(\lambda)+L(\mu).$ 

\begin{theorem} Let $r\le N.$ Define $p_{r,N}(n)$ to be the number of partitions of $n$ into parts $\le N$ where parts $\le r$ appear at least once, and parts $>r$ and $\le N$ appear exactly once if $r<N.$ For $n\equiv2\pmod{24},$ let $I(n)$ be the excess of the number of inequivalent solutions of $n=x^2-3y^2$ in which $x+3y\equiv4\pmod{12},$ over those in which $x+3y\equiv10\pmod{12}.$
Set $\omega(n)=\sum_{N\ge0, 0\le r\le N}(-1)^{N}p_{r,N}(n),$ then $\omega(n)=I(n).$
\end{theorem}

\begin{proof} The partition $\varpi$ described above tells us (1.4) is the generating function for $p_{r,N}(n).$ Therefore, combining (1.3) with the expansion found in [12, pg.77], we have 
\begin{equation}\sum_{n\ge0}\left(\sum_{N\ge0, 0\le r\le N}(-1)^{N}p_{r,N}(n)\right)q^{24n+2}=\sum_{k\in\mathbb{Z}, 2|l|\le k}(-1)^{l+k} q^{3(2k+1)^2-(6l+1)^2}.\end{equation}
The right hand side of (1.5) generates excess of the number of solutions of $2(1-12n)=x^2-3y^2$ where $x\equiv1\pmod{6},$ and $-\frac{3}{2}y<x<\frac{3}{2}y,$ in which $3y+x\equiv4\pmod{12},$ over those in which $3y+x\equiv10\pmod{12}.$ Note that $-\frac{k}{2}\le l \le \frac{k}{2}$ implies 
$$-\frac{3}{2}y<-\frac{3}{2}y+\frac{1}{2}=-\frac{6k}{2}+1\le x \le \frac{6k}{2}+1=\frac{3}{2}y-\frac{1}{2}<\frac{3}{2}y.$$ Setting $D=3,$ $y_1=1,$ $x_1=1$ in Lemma 3 of [2] tells us that each equivalence class of solutions to this equation has a unique $(x,y)$ such that $-\frac{3}{2}y<x<\frac{3}{2}y.$

\end{proof}

\section{Preliminaries for main identities}
A pair of sequences $(\alpha_n(a, q),\beta_n(a,q))$ is called a Bailey pair [17] relative to
$(a,q)$ if
\begin{equation}\beta_n(a,q)=\sum_{0\le j\le n}\frac{\alpha_j(a,q)}{(q;q)_{n-j}(aq;q)_{n+j}}.\end{equation}
\begin{lemma} ([17]) For a pair $(\alpha_n(a; q),\beta_n(a,q))$  satisfying (2.1), we have,
\begin{equation}\sum_{n\ge0}(X)_n(Y)_n(aq/XY)^n\beta_n=\frac{(aq/X)_{\infty}(aq/Y)_{\infty}}{(aq)_{\infty}(aq/XY)_{\infty}}\sum_{n\ge0}\frac{(X)_n(Y)_n(aq/XY)^n\alpha_n}{(aq/X)_n(aq/Y)_n}.\end{equation} \end{lemma}
We also require [6, eq.(S2)], which says that if $(\alpha_n(a; q),\beta_n(a,q))$ is a Bailey pair, then so is $(\alpha_n'(a; q),\beta_n'(a,q))$ where,
\begin{equation}\alpha_n'(a,q)=a^{n/2}q^{n^2/2}\alpha_n(a,q),\end{equation}
\begin{equation}\beta_n'(a,q)=\frac{1}{(-\sqrt{aq};q)_{n}}\sum_{0\le k \le n}\frac{(-\sqrt{aq};q)_k}{(q)_{n-k}}a^{k/2}q^{k^2/2}\beta_{k}(a,q).\end{equation}
A related cousin that will be equally useful is [6, eq.(S5)], which tells us $(\alpha_n'(a; q),\beta_n'(a,q))$ is a Bailey pair where,
\begin{equation}\alpha_n'(a,q)=\frac{(-a^{1/2}q)_n}{(-a^{1/2})_n}a^{n/2}q^{(n^2-n)/2}\alpha_n(a,q),\end{equation}
\begin{equation}\beta_n'(a,q)=\frac{1}{(-\sqrt{a};q)_{n}}\sum_{0\le k \le n}\frac{(-\sqrt{a}q;q)_k}{(q)_{n-k}}a^{k/2}q^{(k^2-k)/2}\beta_{k}(a,q).\end{equation}

Generalizing [4, pg.139, Lemma 6.5.1], we can obtain another identity that will be key in obtaining our general identity.
\begin{lemma} For $-1\le b \le1,$ we have,
$$b^{N+1}\sum_{n\ge0}\frac{(aq)_n}{(bq)_n}q^{(N+1)n}=\frac{1-b}{1-q^{N+1}}\frac{(q)_{N+1}}{(\frac{a}{b}q)_{N+1}}\left(\frac{(aq)_{\infty}}{(b)_{\infty}}-\sum_{N\ge n\ge0}\frac{(\frac{a}{b}q)_nb^n}{(q)_n}\right).$$
\end{lemma}

\begin{proof} From Fine's text [8, pg.5, eq.(6.3), $t=q^{N+1}$] we have
\begin{equation} \begin{aligned} \sum_{n\ge0}\frac{(aq)_n}{(bq)_n}q^{(N+1)n}&=\frac{1-b}{1-q^{N+1}}\sum_{n\ge0}\frac{(\frac{a}{b}q^{N+2})_n}{(q^{N+2})_n}b^n \\
&=\frac{1-b}{1-q^{N+1}}\frac{(q)_{N+1}}{(\frac{a}{b}q)_{N+1}b^{N+1}}\sum_{n\ge0}\frac{(\frac{a}{b}q)_{n+N+1}b^{n+N+1}}{(q)_{n+N+1}b^{N+1}}\\
&=\frac{1-b}{1-q^{N+1}}\frac{(q)_{N+1}}{(\frac{a}{b}q)_{N+1}b^{N+1}}\left(\sum_{n\ge0}\frac{(\frac{a}{b}q)_nb^n}{(q)_n}-\sum_{N\ge n\ge0}\frac{(\frac{a}{b}q)_nb^n}{(q)_n}\right)\\
&=\frac{1-b}{1-q^{N+1}}\frac{(q)_{N+1}}{(\frac{a}{b}q)_{N+1}b^{N+1}}\left(\frac{(aq)_{\infty}}{(b)_{\infty}}-\sum_{N\ge n\ge0}\frac{(\frac{a}{b}q)_nb^n}{(q)_n}\right).\end{aligned}\end{equation}
The result is valid for $b=\pm1$ by analytic continuation.
\end{proof}

 We are now ready to prove our main Bailey pairs.

\begin{theorem} Define the $A_n$ by (2.8). Then pair $(\alpha_n(q, q),\beta_n(q,q))$ is a Bailey pair where,
$$\alpha_n(q,q)=q^{-n(n+1)/2}A_n(q,q,b),$$
$$\beta_n(q,q)=\frac{1}{(-q)_n(bq)_n}\sum_{0\le k \le n}\frac{(b)_{k}}{(q)_k}(-1)^k.$$
Further we have $(\alpha_n(q, q),\beta_n(q,q))$ is a Bailey pair where,
$$\alpha_n(q,q)=\frac{(1+q^{1/2})}{(1+q^{n+1/2})}q^{-n^2/2}A_n(q,q,b),$$
$$\beta_n(q,q)=\frac{1}{(-q^{3/2})_n(bq)_n}\sum_{0\le k \le n}\frac{(b)_{k}}{(q)_k}(-q^{1/2})^k.$$
\end{theorem}
\begin{proof} To obtain this pair, recall the one parameter Bailey pair $(A_n(aq,q,b),B_n(aq,q,b))$ (relative to $(q,q)$) due to Andrews [1, Lemma 6] 
\begin{equation}A_n(aq,q,b)=\frac{(-1)^n(1-aq^{2n+1})a^nq^{n(3n-1)/2}b^n(aq/b)_n}{(1-aq)(bq)_n}\left(1+\sum_{1\le j\le n}\frac{(aq)_{j-1}(1-aq^{2j})(b)_ja^{-j}q^{-j^2}b^{-j}}{(q)_j(aq/b)_j}\right), \end{equation}
\begin{equation}B_n(aq,q,b)=\frac{1}{(bq)_n}.\end{equation}
Setting $a=q$ in (2.3)--(2.4), if we suppose that the left side is given by (2.8)--(2.9) with $a=1,$ we find that we would require the first pair in Theorem 2.3, by the uniqueness of Bailey pairs in conjunction with the identity of Fine [8, pg.17, eq.(15.5)]
\begin{equation}(t)_{\infty}\sum_{n\ge0}\frac{(aq)_n}{(bq)_n}t^n=\sum_{n\ge0}\frac{(-at)^nq^{n(n+1)/2}}{(bq)_n}\sum_{0\le k\le n}\frac{(b)_k}{(q)_k}a^{-k}.\end{equation} To see this, multiply both sides of (2.10) by $(t)_{\infty}^{-1}$ and equate coefficients of $t^n$ to obtain,
\begin{equation}\frac{(aq)_N}{(bq)_N}=\sum_{0\le n \le N}\frac{(-a)^nq^{n(n+1)/2}}{(q)_{N-n}(bq)_n}\sum_{0\le k\le n}\frac{(b)_k}{(q)_k}a^{-k}\end{equation}
and then set $a=-1.$ The second Bailey pair in Theorem 2.3 follows in the same way, with the difference being that we apply (2.5)--(2.6) (with $a=q$) and then put $a=-q^{-1/2}$ in (2.11).
\end{proof}
\section{Main identities}
Having constructed a sufficient number of lemmas in the previous section, we are now ready to state and prove our main identities.

\begin{theorem} We have,
\begin{equation}\frac{(-x)_{\infty}}{2(-q)_{\infty}}\sum_{n\ge0}\frac{(X)_n(Y)_n}{(-q)_n(xq)_n}(q^2/XY)^n\end{equation}
$$+(1-x)\frac{1}{2}\sum_{n\ge0}\frac{(X)_n(Y)_n}{(q^2;q^2)_n}(-q^2/XY)^n\sum_{k\ge0}\frac{(-x)_k}{(-q)_k}q^{(n+1)k}$$
$$=\frac{(q^2/X)_{\infty}(q^2/Y)_{\infty}}{(q^2)_{\infty}(q^2/XY)_{\infty}}\sum_{n\ge0}\frac{(X)_n(Y)_n(q^2/XY)^nq^{-n(n+1)/2}A_n(q,q,x)}{(q^2/X)_n(q^2/Y)_n}.$$
Further, we have that,
\begin{equation}\frac{(-xq^{1/2})_{\infty}}{(-q^{1/2})_{\infty}}\sum_{n\ge0}\frac{(X)_n(Y)_n}{(-q^{3/2})_n(xq)_n}(q^2/XY)^n\end{equation}
$$+\sum_{n\ge0}\frac{(X)_n(Y)_n(xq^{1/2})_{n+1}}{(-q^{3/2})_n(q)_n(xq)_n}(-q^{5/2}/XY)^n\sum_{k\ge0}\frac{(-xq^{1/2})_k}{(-q^{1/2})_k}q^{(n+1)k}$$
$$=(1+q^{1/2})\frac{(q^2/X)_{\infty}(q^2/Y)_{\infty}}{(q^2)_{\infty}(q^2/XY)_{\infty}}\sum_{n\ge0}\frac{(X)_n(Y)_n(q^2/XY)^nq^{-n^2/2}A_n(q,q,x)}{(1+q^{n+1/2})(q^2/X)_n(q^2/Y)_n}.$$
 \end{theorem}

\begin{proof} Setting $b=-1$ and $a=-x/q$ in Lemma 2.2, we obtain

\begin{equation}\sum_{N\ge n\ge0}\frac{(x)_n(-1)^n}{(q)_n}=\frac{(-x)_{\infty}}{2(-q)_{\infty}}+\frac{(x)_{N+1}}{2(q)_{N}}(-1)^N\sum_{n\ge0}\frac{(-x)_n}{(-q)_n}q^{(N+1)n}.\end{equation}
Then, inserting the first Bailey pair in Theorem 2.3 into Lemma 2.1, and using (3.3), we find (3.1). Similary, putting $b=-q^{1/2}$ and $a=-x/q^{1/2}$ in Lemma 2.2, we obtain 
\begin{equation}\sum_{N\ge n\ge0}\frac{(x)_n(-q^{1/2})^n}{(q)_n}=\frac{(-xq^{1/2})_{\infty}}{(-q^{1/2})_{\infty}}+q^{1/2}\frac{(xq^{1/2})_{N+1}}{(1+q^{1/2})(q)_{N}}(-q^{1/2})^N\sum_{n\ge0}\frac{(-xq^{1/2})_n}{(-q^{1/2})_n}q^{(N+1)n}. \end{equation}Then, inserting the second Bailey pair in Theorem 2.3 into Lemma 2.1, and using (3.4), we find (3.2).

\end{proof}
A nice consequence of Theorem 3.1 include identities for Mock theta functions.

\begin{theorem} We have for $-1\le x\le 1,$
\begin{equation} \frac{1}{2}\frac{(-x)_{\infty}}{(-q)_{\infty}}\sum_{n\ge0}\frac{q^{n(n+1)}}{(-q)_n(xq)_n}+(1-x)\frac{1}{2}\sum_{n\ge0}\frac{(-1)^nq^{n(n+1)}}{(q^2;q^2)_n}\sum_{k\ge0}\frac{(-x)_k}{(-q)_k}q^{(n+1)k}\end{equation}
$$=\frac{1}{(q)_{\infty}}\sum_{n\ge0}q^{n^2/2+n/2}A_n(q,q,x)=\frac{1}{(-q)_{\infty}}\sum_{n\ge0}\frac{(-q)_n}{(xq)_{n}}q^{n(n+1)/2}.$$

\begin{equation} \frac{1}{2(-q)_{\infty}}\sum_{n\ge0}\frac{q^{n(n+1)}}{(-q)_n}+\frac{1}{2}\sum_{n\ge0}\frac{q^n(q^{n+2};q^2)_{\infty}}{(-q)_n}=\frac{1}{(-q)_{\infty}}\sum_{n\ge0}(-q)_nq^{n(n+1)/2}.\end{equation}

\begin{equation} \frac{1}{2}\frac{(-q;q^2)_{\infty}}{(-q^2;q^2)_{\infty}}\sum_{n\ge0}\frac{q^{2n(n+1)}}{(-q)_{2n+1}}+\frac{1}{2}\sum_{n\ge0}\frac{(-1)^nq^{2n(n+1)}}{(q^4;q^4)_n}\sum_{k\ge0}\frac{(-q;q^2)_k}{(-q^2;q^2)_k}q^{2(n+1)k}\end{equation}
$$=\frac{1}{(-q^2;q^2)_{\infty}}\sum_{n\ge0}\frac{(-q^2;q^2)_n}{(-q;q^2)_{n+1}}q^{n(n+1)}.$$

\begin{equation}\frac{(-xq^{1/2})_{\infty}}{(-q^{1/2})_{\infty}}\sum_{n\ge0}\frac{q^{n(n+1)}}{(-q^{1/2})_{n+1}(xq)_n}\end{equation}
$$+q^{1/2}\sum_{n\ge0}\frac{(xq^{1/2})_{n+1}}{(-q^{1/2})_{n+1}(q)_n(xq)_n}(-1)^nq^{n(n+3/2)}\sum_{k\ge0}\frac{(-xq^{1/2})_k}{(-q^{1/2})_k}q^{(n+1)k}$$
$$=\frac{1}{(q^2)_{\infty}}\sum_{n\ge0}\frac{q^{n^2/2+n}A_n(q,q,x)}{(1+q^{n+1/2})}=\frac{1}{(-q^{1/2})_{\infty}}\sum_{n\ge0}\frac{(-q^{1/2})_n}{(xq)_{n}}q^{n(n+2)/2}.$$

\begin{equation}\frac{1}{(-q;q^2)_{\infty}}\sum_{n\ge0}\frac{q^{2n(n+1)}}{(-q;q^2)_{n+1}}+q\sum_{n\ge0}\frac{(-1)^nq^{n(2n+3)}}{(-q;q^2)_{n+1}(q^2;q^2)_n}\sum_{k\ge0}\frac{q^{2(n+1)k}}{(-q;q^2)_k}\end{equation}
$$=\frac{1}{(-q;q^2)_{\infty}}\sum_{n\ge0}(-q;q^2)_nq^{n(n+2)}.$$

\begin{equation} \frac{(q;q^2)_{\infty}}{(-q;q^2)_{\infty}}\sum_{n\ge0}\frac{q^{2n(n+1)}}{(-q)_{2n+1}}+q\sum_{n\ge0}\frac{(-1)^nq^{n(2n+3)}}{(q^4;q^4)_n}\sum_{k\ge0}\frac{(q;q^2)_k}{(-q;q^2)_k}q^{2(n+1)k}\end{equation}
$$=\frac{1}{(-q;q^2)_{\infty}}\sum_{n\ge0}\frac{(-q;q^2)_n}{(-q^2;q^2)_{n}}q^{n(n+2)}.$$

\begin{equation}\frac{(-q^2;q^2)_{\infty}}{(-q;q^2)_{\infty}}\sum_{n\ge0}\frac{q^{2n(n+1)}}{(-q^;q^2)_{n+1}^2}\end{equation}
$$+q\sum_{n\ge0}\frac{(-q^2;q^2)_{n+1}}{(-q;q^2)_{n+1}^2(q^2;q^2)_n}(-1)^nq^{n(2n+3)}\sum_{k\ge0}\frac{(-q^2;q^2)_k}{(-q;q^2)_k}q^{2(n+1)k}$$
$$=\frac{1}{(-q;q^2)_{\infty}}\sum_{n\ge0}\frac{q^{n(n+2)}}{1+q^{2n+1}}.$$

\end{theorem}

\begin{proof} For (3.5) Let $X,Y\rightarrow\infty$ in (3.1), and then invoke the transformation through the indefinite theta expansion [1]. Equation (3.6) is the $x=0$ case of (3.5). Equation (3.7) is the $x=-q^{1/2}$ case of (3.5). For (3.8) Let $X,Y\rightarrow\infty$ in (3.2), and then invoke the indefinite theta expansion [1]. Equation (3.9) is the $x=0$ case of (3.8). Equation (3.11) is the $x=-q^{1/2}$ case of (3.8).

\end{proof}
First, we remark on our \begin{equation}\sum_{n\ge0}\frac{q^{n(n+1)}}{(-q)_n(xq)_n}\sum_{0\le k\le n}\frac{(x)_k}{(q)_k}(-1)^k,\end{equation}
and
\begin{equation}\sum_{n\ge0}\frac{q^{n(n+1)}}{(-q^{1/2})_{n+1}(xq)_n}\sum_{0\le k\le n}\frac{(x)_k}{(q)_k}(-1)^k.\end{equation}
Equation (3.12) is mixed mock modular forms for $x=0, -q^{1/2},$ since multiplying by $(-q)_{\infty}$ gives a mock modular form, and for $x=-1$ we have a modular form, $(q^2;q^2)_{\infty}.$ 
Similarly, equation (3.13) is mixed mock modular forms for $x=0, -1,$ since multiplying by $(-q^{1/2})_{\infty}$ gives a mock modular form, and for $x=-q^{-1/2}$ we have a modular form $q^{-1}((q^2;q^2)_{\infty}/(q;q^2)_{\infty}-1)/(-q^{1/2})_{\infty}.$ Equation (3.6) provides a connection between two mock theta functions of order five [1]. Equation (3.7) provides a connection between the eighth order mock theta a function $T_1(q)$ [9], and one Mock theta function found in [5]. Equation (3.9) related two mock theta functions of order five [1]. Equation (3.10) provides a connection between the eighth order mock theta a function $S_1(q)$ [9], and one Mock theta function found in [5]. Equation (3.11) involves a mock theta function of order three [8, pg.55].

1390 Bumps River Rd. \\*
Centerville, MA
02632 \\*
USA \\*
E-mail: alexpatk@hotmail.com, alexepatkowski@gmail.com
\end{document}